\documentclass[12pt, reqno]{amsart}
\usepackage{amsmath, amsthm, amscd, amsfonts, amssymb, graphicx, color}
\usepackage[english]{babel}
\usepackage[bookmarksnumbered, colorlinks, plainpages]{hyperref}

\newtheorem{theorem}{Theorem}[section]

\newtheorem{proposition}[theorem]{Proposition}
\newtheorem{corollary}[theorem]{Corollary}
\theoremstyle{definition}

\theoremstyle{remark}
\newtheorem{remark}[theorem]{Remark}
\numberwithin{equation}{section}
\begin{document}
\setcounter{page}{1}

\title[Automorphisms  of central extensions ...]{Automorphisms
  of central extensions  of type I von Neumann algebras}

\author[S.~Albeverio, Sh.~A.~AYUPOV, K.~K.~KUDAYBERGENOV,
R.~T.~Djumamuratov]{S.~Albeverio$^1$, Sh.~A.~AYUPOV,$^2$  K.~K.~KUDAYBERGENOV
$^3$ and R.~T.~Djumamuratov$^4$}

\address{$^{1}$ Institut f\"{u}r Angewandte Mathematik and HCM,
Rheinische Friedrich-Wilhelms-Uni\-versit\"{a}t Bonn, Endenicher
Allee 60,  D-53115 Bonn, Germany}

\email{\textcolor[rgb]{0.00,0.00,0.84}{albeverio@uni-bonn.de}}

\address{$^{2}$ Department of Algebra and analysis, Institute of
 Mathematics and Information  Technologies, Uzbekistan Academy of Sciences  \\
Dormon yoli 29, 100125,  Tashkent,   Uzbekistan}

\email{\textcolor[rgb]{0.00,0.00,0.84}{sh\_ayupov@mail.ru}}

\address{$^{3}$ Department of Mathematics, Karakalpak state university\\
Ch. Abdirov 1,  230113, Nukus,    Uzbekistan
\newline}

\email{\textcolor[rgb]{0.00,0.00,0.84}{karim2006@mail.ru}}

\address{$^{4}$ Department of Mathematics, Karakalpak state university\\
Ch. Abdirov 1,  230113, Nukus,    Uzbekistan
\newline}

\email{\textcolor[rgb]{0.00,0.00,0.84}{rauazh@mail.ru}}






\begin{abstract}
Given a von Neumann algebra $M$ we  consider the  central
exten\-sion $E(M)$ of $M.$  For type I von Neumann algebras $E(M)$
coincides with the algebra $LS(M)$ of all locally measurable
operators affiliated with $M.$ In this case  we show that an
arbitrary automorphism $T$ of $E(M)$ can be decomposed as
$T=T_a\circ T_\phi,$ where $T_a(x)=axa^{-1}$ is an inner
automorphism implemented by an element $a\in E(M),$ and $T_\phi$
is a special automorphism generated by an automorphism $\phi$ of
the center of $E(M).$ In particular if $M$ is of type I$_\infty$
then every band  preserving automorphism  of $E(M)$ is inner.
\end{abstract} \maketitle

\section{Introduction}

In the series of paper \cite{Alb1}-\cite{AK1}  we have considered
derivations on the algebra $LS(M)$ of locally measurable operators
affiliated with a von Neumann algebra $M,$ and on various
subalgebras of $LS(M).$ A complete description of derivations has
been obtained in the case of  von Neumann algebras of type I and
III.

A comprehensive survey of recent results concerning derivations on
various algebras of unbounded operators affiliated with von
Neumann algebras is presented in \cite{AK2}.

 It is well-known that properties
of derivations on algebras are strongly correlated with properties
of automorphisms of underlying algebras (see e.g. \cite{KR67}).
Algebraic automorphisms of $C^\ast$-algebras and von Neumann
algebras were considered in the paper of R. Kadison and J.
Ringrose \cite{KR74}, which is devoted to automatic continuity and
innerness of automorphisms. By this paper we initiate a study of
automorphisms of the algebra $LS(M)$ and its various subalgebras.
In the commutative case a similar problem has been considered by
A.G.~Kusraev \cite{Kus} who proved by means of Booolean-valued
analysis the existence of non trivial band preserving automorphism
on algebras of the form $L^0(\Omega, \Sigma, \mu).$ The algebra
$LS(M)$ and its subalgebras present a non commutative counterparts
of the algebra $L^0(\Omega, \Sigma, \mu).$ In the present paper we
establish a general form of automorphisms of the algebra $LS(M)$
for type I von Neumann algebras $M.$

Let   $\mathcal{A}$ be an algebra. A one-to-one
 linear operator
$T:\mathcal{A}\rightarrow \mathcal{A}$ is called an
\textit{automorphism} if $T(xy)=T(x)T(y)$  for all $x, y\in
\mathcal{A}.$ Given an invertible  element  $a\in A$ one can
define an automorphism $T_a$ of $\mathcal{A}$
 by  $T_a(x)=axa^{-1},\,x\in \mathcal{A}.$ Such automorphisms are called
   \emph{inner automorphisms} of $\mathcal{A}.$
   It is clear that for commutative (abelian)
algebra $\mathcal{A}$  all inner automorphisms are trivial, i.e.
acts as unit operator. In the general case inner automorphisms are
identical on the center of $\mathcal{A}.$ Essentially different
classes of automorphisms are those which are generated by
automorphisms of the center $Z(\mathcal{A})$ of $\mathcal{A}.$ In
some cases such automorphisms $\phi$ on $Z(\mathcal{A})$ can be
extended to automorphisms $T_\phi$ of the whole algebra
$\mathcal{A}$ (see e.g. Kaplansky \cite[Theorem 1]{Kap52}). The
main result of the present paper shows that for a type I von
Neumann algebra $M$ every automorphism $T$ of the algebra $LS(M)$
can be uniquely decomposed as a composition $T=T_a\circ T_\phi$ of
an inner automorphism $T_a$ and an automorphism $T_\phi$ generated
by an automorphism $\phi$ of the center of $LS(M).$

In section 2  we recall the notions of the algebras $S(M)$ of
measurable operators and $LS(M)$ of locally measurable operators
affiliated with a von Neumann algebra $M.$ We also introduce the
so-called \emph{central extension} $E(M)$ of the von Neumann
algebra $M.$ In the general case $E(M)$ is a *-subalgebra of
$LS(M),$ which coincides with $LS(M)$ if and only if $M$ does not
have direct summands of type II. We also introduce two
generalizations of the topology of convergence locally in measure
on $LS(M)$ and prove that for the type I case they coincide.

In section 3 we consider automorphisms of the algebra $E(M)$ --
the central extension of a von Neumann algebra $M.$ We prove
(Theorem \ref{A3}) that if $M$ is of the type I then each
automorphism $T$ of $E(M)$ which acts identically on the center
$Z(E(M))$ of $E(M),$ is inner. We also show that for homogeneous
type I  von Neumann algebras $M$ every automorphism $\phi$ of the
center $Z(E(M))$ of $E(M)$ can be extended to an automorphism
$T_\phi$ of the whole $E(M).$ Finally we prove the main result of
the present paper which shows that each automorphism $T$ of $E(M)$
for a type I von Neumann algebra $M$ can be uniquely represented
as $T=T_a\circ T_\phi,$ where $T_a$ is an inner automorphism
implemented by an element $a\in E(M),$ and $T_\phi$ is an
automorphism generated by an automorphism $\phi$ of the center of
$E(M).$ In particular we obtain that each  bundle preserving
automorphism of $E(M)$ is inner if $M$ is of type I$_\infty.$

\section{Central extensions of von Neumann algebras}

In this section we give some necessary definitions and a
preliminary information concerning algebras of measurable and
locally measurable operators affiliated with a von Neumann
algebra. We also introduce the notion of the central extension of
a von Neumann algebra.

Let  $H$ be a complex Hilbert space and let  $B(H)$ be the algebra
of all bounded linear operators on   $H.$ Consider a von Neumann
algebra $M$  in $B(H)$ with the operator norm $\|\cdot\|_M.$ Denote by
$P(M)$ the lattice of projections in $M.$

A linear subspace  $\mathcal{D}$ in  $H$ is said to be
\emph{affiliated} with  $M$ (denoted as  $\mathcal{D}\eta M$), if
$u(\mathcal{D})\subset \mathcal{D}$ for every unitary  $u$ from
the commutant
$$M'=\{y\in B(H):xy=yx, \,\forall x\in M\}$$ of the von Neumann algebra $M.$

A linear operator  $x$ on  $H$ with the domain  $\mathcal{D}(x)$
is said to be \emph{affiliated} with  $M$ (denoted as  $x\eta M$) if
$\mathcal{D}(x)\eta M$ and $u(x(\xi))=x(u(\xi))$
 for all  $\xi\in
\mathcal{D}(x).$

A linear subspace $\mathcal{D}$ in $H$ is said to be \emph{strongly
dense} in  $H$ with respect to the von Neumann algebra  $M,$ if

1) $\mathcal{D}\eta M;$

2) there exists a sequence of projections
$\{p_n\}_{n=1}^{\infty}$ in $P(M)$  such that
$p_n\uparrow\textbf{1},$ $p_n(H)\subset \mathcal{D}$ and
$p^{\perp}_n=\textbf{1}-p_n$ is finite in  $M$ for all
$n\in\mathbb{N},$ where $\textbf{1}$ is the identity in $M.$

A closed linear operator  $x$ acting in the Hilbert space $H$ is said to be
\emph{measurable} with respect to the von Neumann algebra  $M,$ if
 $x\eta M$ and $\mathcal{D}(x)$ is strongly dense in  $H.$ Denote by
 $S(M)$ the set of all measurable operators with respect to
 $M$ (see \cite{Seg}).

A closed linear operator $x$ in  $H$  is said to be \emph{locally
measurable} with respect to the von Neumann algebra $M,$ if $x\eta
M$ and there exists a sequence $\{z_n\}_{n=1}^{\infty}$ of central
projections in $M$ such that $z_n\uparrow\textbf{1}$ and $z_nx \in
S(M)$ for all $n\in\mathbb{N}$ (see \cite{Yea}).

It is well-known  \cite{Mur}, \cite{Yea} that the set $LS(M)$ of
all locally measurable operators with respect to $M$ is a unital
*-algebra when equipped with the algebraic operations of strong
addition and multiplication and taking the adjoint of an operator,
and contains $S(M)$ as a solid *-subalgebra.

Let $(\Omega,\Sigma,\mu)$  be a measure space and from now on
suppose
 that the measure $\mu$ has the  direct sum property, i. e. there is a family
 $\{\Omega_{i}\}_{i\in
J}\subset\Sigma,$ $0<\mu(\Omega_{i})<\infty,\,i\in J,$ such that
for any $A\in\Sigma,\,\mu(A)<\infty,$ there exist a countable
subset $J_{0 }\subset J$ and a set  $B$ with zero measure such
that $A=\bigcup\limits_{i\in J_{0}}(A\cap \Omega_{i})\cup B.$

 We denote by  $L^{0}(\Omega, \Sigma, \mu)$ the algebra of all
(equivalence classes of) complex measurable functions on $(\Omega,
\Sigma, \mu)$ equipped with the topology of convergence in
measure.

Consider the algebra  $S(Z(M))$  of operators which are measurable
with respect to the  center $Z(M)$ of the von Neumann algebra $M.$
Since  $Z(M)$ is an abelian von Neumann algebra  it
 is *-isomorphic to $ L^{\infty}(\Omega, \Sigma, \mu)$
   for an appropriate measure space $(\Omega, \Sigma, \mu)$. Therefore the algebra  $S(Z(M))$ coincides
   with $Z(LS(M))$ and  can be
 identified with the algebra $ L^{0}(\Omega, \Sigma, \mu)$ of all
 measurable functions on $(\Omega, \Sigma, \mu)$.

The basis of neighborhoods of zero in the topology of convergence locally in measure
    on $L^0(\Omega,\Sigma, \mu)$ consists of the sets
$$W(A,\varepsilon,\delta)=\{f\in L^0(\Omega,\Sigma, \mu):\exists B\in \Sigma, \, B\subseteq A, \,
\mu(A\setminus B)\leq \delta, $$
$$ f\cdot \chi_B \in L^{\infty}(\Omega,\Sigma, \mu),\,
||f\cdot \chi_B||_{L^{\infty}(\Omega,\Sigma, \mu)}\leq \varepsilon\},$$
where $\varepsilon, \delta>0, \, A\in \Sigma, \, \mu(A)<+\infty,$
and $\chi_B$ is the characteric
function of the set $B\in \Sigma.$

Recall the definition of the dimension functions on the lattice
$P(M)$ of projection from $M$ (see \cite{Mur}, \cite{Seg}).

By  $L_+$ we denote the set of all measurable  functions  $f:
(\Omega,\Sigma, \mu)\rightarrow [0,{\infty}]$ (modulo functions
equal to zero $\mu$-almost everywhere ).

 Let  $M$ be an arbitrary von Neumann algebra with the center $Z=L^\infty(\Omega,\Sigma, \mu).$
Then there exists a map  $D:P(M)\rightarrow
L_{+}$ with the following properties:

(i) $d(e)$ is a finite function if only if the projection  $e$ is finite;

(ii) $d(e+q)=d(e)+d(q)$  for $p, q \in P(M),$ $eq=0;$

(iii) $d(uu^*)=d(u^*u)$ for every  partial isometry  $u\in M;$

(iv) $d(ze)=zd(e)$ for all $z\in P(Z(M)), \,\, e\in P(M);$

(v) if  $\{e_{\alpha}\}_{\alpha \in J}, \,\,\, e\in P(M) $ and $e_{\alpha}\uparrow e,$ then
$$d(e)=\sup \limits_{\alpha \in J}d(e_{\alpha}).$$
This map  $d:P(M)\rightarrow L_+,$ is a called the \emph{dimension functions} on  $P(M).$

\begin{remark} \label{R}
Recall that for an element   $x\in M$ the projection defined as
$$
c(x)=\inf\{z\in P(Z(M)): zx=x\}
$$
 is called the  central
cover of $x.$

Let $M$ be a type I von Neumann algebra. If  $p, q\in P(M)$ are
abelian projections with $c(p)=c(q)=\textbf{1},$ then the property
(iii) implies  that $0<d(p)(\omega)=d(q)(\omega)<\infty$ for
$\mu$-almost every  $\omega\in\Omega.$ Therefore replacing $d$ by
$d(p)^{-1}d$ we can assume that  $d(p)=c(p)$ for every abelian
projection  $p\in P(M).$ Thus for all  $e\in P(M)$ we have that
$d(e)\geq c(e).$
\end{remark}

The basis of neighborhoods of zero in the topology $t(M)$ of \emph{convergence
locally in measure}  on $LS(M)$ consists (in the above notations)
of the following sets
$$V(A,\varepsilon,\delta)=\{x\in LS(M):\exists p\in P(M), \, \exists z\in P(Z(M)),  \,
xp \in M, $$
$$ ||xp||_{M}\leq \varepsilon, \,\, z^{\bot}\in W(A,\varepsilon,\delta), \,\, d(zp^{\bot})\leq \varepsilon z\},$$
where
 $\varepsilon, \delta>0, \, A\in \Sigma, \, \mu(A)<+\infty.$

The topology  $t(M)$ is  metrizable if and only if the center  $Z(M)$
is   $\sigma$-finite (see \cite{Mur}).

Given an arbitrary  family  $\{z_i\}_{i\in I}$ of mutually orthogonal
central projections in $M$ with $\bigvee\limits_{i\in
I}z_i=\textbf{1}$ and a  family of elements $\{x_i\}_{i\in I}$ in
$LS(M)$ there exists a unique element $x\in LS(M)$ such that $z_i
x=z_i x_i$ for all $i\in I.$ This element is denoted by
$x=\sum\limits_{i\in I}z_i x_i.$

We  denote by  $E(M)$  the set of all elements  $x$ from  $LS(M)$ for which there exists a sequence of
mutually orthogonal central projections  $\{z_i\}_{i\in I}$ in  $M$ with $\bigvee\limits_{i\in I}z_i=\textbf{1},$
such that $z_i x\in M$ for all $i\in I,$ i.e.
 $$E(M)=\{x\in LS(M): \exists z_i\in P(Z(M)), z_iz_j=0, i\neq j, \bigvee\limits_{i\in I}z_i=\textbf{1},
 z_i x\in M, i\in I\},$$
where $Z(M)$ is the center of $M.$

It is known  \cite{AK1} that  $E(M)$ is  *-subalgebras in  $LS(M)$ with the center
 $S(Z(M)),$ where   $S(Z(M))$ is  the algebra of all measurable operators
 with respect to   $Z(M),$ moreover,
  $LS(M)=E(M)$ if and only if $M$ does not have
direct summands of type II.

A similar notion (i.e. the algebra $E(\mathcal{A})$) for
arbitrary *-subalgebras $\mathcal{A}\subset LS(M)$  was independently
introduced recently by M.A. Muratov and V.I. Chilin \cite{Mur1}.
The algebra  $E(M)$ is called
\textit{the central extension of} $M.$

It is known (\cite{AK1},
\cite{Mur1}) that an element
$x\in  LS(M)$ belongs to $E(M)$ if and only if there exists
 $f\in S(Z(M))$ such that    $|x|\leq f.$
Therefore for each
 $x\in E(M)$ one can define the following vector-valued norm
 \begin{equation}
 \label{norm}
 ||x||=\inf\{f\in S(Z(M)): |x|\leq f\}
\end{equation}
and this norm satisfies the following conditions:

$1) \|x\|\geq 0; \|x\|=0\Longleftrightarrow x=0;$

$2)  \|f x\|=|f|\|x\|;$

$ 3)  \|x+y\|\leq\|x\|+\|y\|;$

$4) ||x y||\leq ||x||||y||;$

$5) ||xx^{\ast}||=||x||^2$
\newline
    for all $x,y\in E(M), f\in S(Z(M)).$

Let us equip  $E(M)$ with the topology which is
defined by the following system of zero neighborhoods:
$$
O(A, \varepsilon, \delta)=\left\{x \in E(M): ||x||\in W(A, \varepsilon, \delta)\right\},
$$
where  $\varepsilon ,\delta > 0,\,\,\,A \in \sum ,\,\,\,\mu \left( A \right) <
+ \infty.$

Denote the above topology by  $t_c(M).$

\begin{proposition}\label{A}
The topology $t_c(M)$ is stronger that
the topology $t(M)$ of convergence locally in measure.
\end{proposition}

\begin{proof}  It is sufficient to show that
\begin{equation}
 \label{baza}
O(A, \varepsilon, \delta) \subset V(A, \varepsilon, \delta).
\end{equation}

Let  $x \in O(A, \varepsilon, \delta),$
i.e. $||x|| \in W(A, \varepsilon, \delta).$
Then there exists  $B \in \Sigma$ such that
 \[
 B \subseteq A,\,\,\,\mu(A\setminus
  B) \le \delta,
  \]
 and
\[
||x||\chi_B \in L^\infty(\Omega, \Sigma, \mu),\,\,||\|x\|\chi_B||_{M}\le \varepsilon.
\]

Put  $z =p=\chi_B.$
Then  $||xp||=||x\chi_B||=||x||\chi_B\in L^\infty(\Omega, \Sigma, \mu),$ i.e. $xp
\in M$ and moreover $||xp||_M \le \varepsilon.$ Since  $\mu(A\setminus B) \le \delta$ and
$z^\perp \chi_B = \chi _B^\perp \chi_B =
0,$ one has $z^\perp\in W(A, \varepsilon, \delta).$ Therefore
\[
||xp||_M\leq\varepsilon,\,\,z^\perp\in W(A, \varepsilon, \delta),\,\, zp^\perp=\chi _B \chi_B^\perp=0
\]
and hence
$
x \in V(A, \varepsilon, \delta).
$
\end{proof}

\begin{proposition}\label{B}
If $M$ is a type  I von Neumann algebra and $0<\varepsilon<1,$ then
\[
O(A, \varepsilon, \delta)=V(A, \varepsilon, \delta).
\]
\end{proposition}

\begin{proof} From above  \eqref{baza} we have that
$O(A, \varepsilon, \delta) \subset V(A, \varepsilon, \delta).$
Therefore it is sufficient to show that  $V(A, \varepsilon, \delta)
\subset O (A, \varepsilon, \delta).$

Let  $x \in V(A, \varepsilon, \delta).$ Then there exist  $p
\in P(M)$ and $z \in P(Z(M))$
such that
\[
xp \in M,\,\,\,\,||xp||_M \le \varepsilon
,\,\,\,\,z^\perp \in W(A, \varepsilon, \delta),\,\,\,d(zp^\perp)\le \varepsilon z.
\]

Since   $M$ is of type  I Remark \ref{R} implies that  $d(zp^\perp)\geq c(zp^\perp).$
Now from   $d(zp^\perp) \le \varepsilon z$ it follows that
 $c(zp^\perp)\leq\varepsilon z.$
From  $0<\varepsilon<1$ we obtain that    $zp^\perp=0.$
Therefore   $z \le c(p),$ where  $c(p)$ is the central cover of  $p.$ Thus
 $z = zp.$
Put $z=\chi_E$ for an appropriate $E\in \Sigma.$ Since
 $z^\perp \in W(A, \varepsilon, \delta)$
one has that  $\chi_{\Omega \setminus E} \in W(A, \varepsilon, \delta).$
  Thus there exists  $B \in \Sigma $ such that  $B \subseteq
A,\,\,\,\,\mu(A\setminus B) \le \delta,$ $|\chi_{\Omega \setminus E}\chi _B| \le \varepsilon < 1.$
 Hence   $\chi_B\leq \chi_E.$ So we obtain
\[
||x||\chi_B \le ||x||\chi_E=||x||z=||xz||=||xzp||=||xp|| \le \varepsilon.
\]
This means that  $x \in O(A, \varepsilon, \delta).$
\end{proof}

\begin{corollary}\label{C}
If  $M$ is a type I von Neumann algebra then the topologies $t(M)$ and $t_c(M)$
coincide.
\end{corollary}

\begin{proposition}\label{D}
Let  $M$ be a type I von Neumann algebra and
  $x\in LS(M),$ $x\geq 0.$ If
 $pxp=0$ for all abelian projections $p\in M$ then $x=0.$
\end{proposition}

\begin{proof}   Since $x\geq 0$ we have that  $x=yy^{\ast}$ for an appropriate
 $y\in LS(M).$  Then
$$
0=pxp=pyy^{\ast}p=py(py)^{\ast}$$ and hence    $py=0.$ Therefore
$y^{\ast}py=0$ for all abelian projections $p\in M.$ But since $M$
has the type I there exists a family $\{p_i\}_{i\in J}$ of
mutually orthogonal abelian projections
   such that  $\sum \limits_{i\in
J}p_i= \mathbf{1}.$ For any finite subset  $F\subseteq
J$ put  $p_F=\sum \limits_{i\in F}p_i.$ Since  $p_F\uparrow
\mathbf{1}$ from   $yp_Fy^{\ast}=0$ we have that  $yy^{\ast}=0,$
i.e. $x=yy^{\ast}=0.$
\end{proof}

\section{Automorphisms of central extensions for type I
von Neumann algebras}

Let  $\mathcal{A}$ be an arbitrary algebra with
the center  $Z(\mathcal{A})$ and let  $T:\mathcal{A}\rightarrow \mathcal{A}$
be an automorphism.
It is clear that  $T$ maps  $Z(\mathcal{A})$ onto itself. Indeed
for all
 $a\in Z(\mathcal{A})$ and  $x\in \mathcal{A}$
one has
$$
 T(a)T(x)=T(ax)=T(xa)=T(x)T(a)
 $$
which means that  $T(a)\in Z(\mathcal{A}).$

An operator  $T:\mathcal{A}\rightarrow \mathcal{A}$ is said to be
  $Z(\mathcal{A})$-linear if  $T(ax)=aT(x)$ for all $a\in Z(\mathcal{A})$ and
$x\in \mathcal{A}.$
It is easy to see that an automorphism
$T:\mathcal{A}\rightarrow \mathcal{A}$
of a unital algebra $\mathcal{A}$ is  $Z(\mathcal{A})$-linear if and only if it is identical
on the center $Z(\mathcal{A}).$

\begin{theorem}\label{A1}
Let   $M$ be a von Neumann algebra of type I and let $E(M)$ be its central extension.
Then each $Z(E(M))$-linear automorphism $T$ of the algebra $E(M)$ is inner.
\end{theorem}

\begin{proof} Let us show that  $T$ is  $t(M)$-continuous. First
suppose that the center  $Z(M)$ of the von Neumann algebra $M$
is  $\sigma$-finite.
Then the topology  $t(M)$ is metrizable and hence it is sufficient
to prove that the operator $T$ is $t(M)$-closed.

Consider a sequence  $\{x_n\}\subset E(M)$
such that $x_n\stackrel{t(M)}\longrightarrow 0,$
$T(x_n)\stackrel{t(M)}\longrightarrow y.$
Take  $x\in E(M)$ such that  $T(x)=y$
and let us show that  $x=0.$ Since
$$
x^{\ast}x_n\stackrel{t(M)}\longrightarrow 0
$$
and
$$
T(x^{\ast}x_n)=T(x^{\ast})T(x_n)\stackrel{t(M)}\longrightarrow
T(x^{\ast})y=T(x^{\ast})T(x)=T(x^{\ast}x),
$$
we may suppose (by replacing the sequence  $\{x_n\}$ by the
sequence  $\{x^{\ast}x_n\}$) that $x\geq 0.$

Let  $p\in M$ be an arbitrary abelian projection with
$c(p)=\textbf{1}.$
Then  $px_n p=a_n p$ for an appropriate
$a_n\in S(Z(M)),$  $n\in \mathbb{N}.$ Since  $x_n\stackrel{t(M)}\longrightarrow 0$
and  $c(p)=\textbf{1}$
it follows that  $a_n\stackrel{t(M)}\longrightarrow 0.$
Therefore
$$
T(p)T(x_n)T(p)=T(px_np)=T(a_np)=a_nT(p)\stackrel{t(M)}\longrightarrow 0.
$$
On the other hand
$$
T(p)T(x_n)T(p)\stackrel{t(M)}\longrightarrow T(p)yT(p),
$$
thus  $T(p)yT(p)=0$ and hence
$$
pxp=T^{-1}(T(p)yT(p))=T(0)=0,
$$
i.e. $pxp=0$ for all abelian projections with $c(p)=\textbf{1}.$
Therefore Proposition  \ref{D}  implies that  $x=0,$
i.e.  $T$ is  $t(M)$-continuous.

Now consider the general case, i.e. when the center
 $Z(M)$ is arbitrary.
Take a family  $\{z_i\}_{i\in I}$ of mutually orthogonal central
projections in  $M$ with $\bigvee\limits_{i}z_i=\textbf{1}$ such
that  $z_iZ(M)$ is $\sigma$-finite for all $i\in I.$ From the
above we have that $z_i T$ is $t(z_i M)$ continuous on $z_i E(M)$
for all $i\in I,$ where $(z_i T)(x)=T(z_i x)=z_iT(x)$ is the
restriction of $T$ onto $z_i E(M)$ which is well-defined in view
of the $Z(E(M))$-linearity of $T.$ Therefore $T$ is
$t(M)$-continuous of whole $E(M)=\bigoplus\limits_{i\in
I}z_iE(M).$

Further by Corollary  \ref{C} the topologies $t(M)$ and $t_c(M)$
coincide and hence $T$ is also $t_c(M)$-continuous and according
to \cite[Theorem 2]{Zak} there exists $c\in S(Z(M))$ such that
$||T(x)||\leq c ||x||$ for all $x\in E(M).$

Take a sequence
  $\{z_n\}_{n\in \mathbb{N}}$ of mutually orthogonal central projections in
    $M$ with   $\bigvee\limits_{n}z_n=\textbf{1}$
such that   $z_n c\in Z(M)$ for all $n\in \mathbb{N}.$ This means
that the automorphism $z_nT$ maps bounded elements from $z_nE(M)$
to bounded elements, i.e. $z_nT(z_nM)\subseteq z_nM.$ Then given
any  $n\in \mathbb{N}$ the automorphism  $z_n T|_{z_n M}$ is
identical on the center of  $z_n M.$ By theorem of Kaplansky
\cite[Theorem 10]{Kap} there exist elements  $a_n \in z_n M$ which
are invertible in $z_n M,$ such that $z_nT(x)=a_n x a_n^{-1}$ for
all $x\in z_n M.$ Put $a=\sum\limits_{n\geq 1}z_n a_n.$ It is
clear that $a\in E(M)$ and
$$
T(x)=\sum\limits_{n\geq 1}z_n T(x)=\sum\limits_{n\geq 1}z_n T(z_nx)
=\sum\limits_{n\geq 1}a_n (z_nx)a_n=axa^{-1}
$$ for all
$x\in E(M).$ \end{proof}

Let  $M$ be a von Neumann algebra of type I$_n,$ $n\in
\mathbb{N},$ with the center  $Z(M).$ Then  $M$ is *-isomorphic to
the algebra  $M_n(Z(M))$ of all $n\times n$ matrices over  $Z(M)$
(cf. \cite[Theorem  2.3.3]{Sak}). Moreover   the algebra
$S(M)=E(M)$ is *-isomorphic to the algebra
 $M_n(Z(S(M))),$ where
$Z(S(M))=S(Z(M))$ is the center of  $S(M)$ (see \cite[Proposition
1.5]{Alb2}). If  $e_{i j},\,i,j=\overline{1, n}$ are matrix units
in $M_n(S(Z(M)))$ then each element  $x\in M_n(S(Z(M)))$ has the
form
 $$x=\sum\limits_{i,j=1}^{n}a_{i j}e_{i j},\,a_{i j}\in S(Z(M)),\,i,j=\overline{1, n}.$$
Let   $\phi:S(Z(M))\rightarrow S(Z(M))$ be an automorphism. Setting
\begin{equation}
\label{cenfin}
T_{\phi}\left(\sum\limits_{i,j=1}^{n}a_{i j}e_{i j}\right)=
 \sum\limits_{i,j=1}^{n}\phi(a_{i j})e_{i j}
\end{equation}
we obtain a linear operator  $T_\phi$ on  $M_n(S(Z(M))),$
which is in fact an automorphism of  $M_n(S(Z(M))).$ Indeed, for
$$
x=\sum\limits_{i,j=1}^{n}a_{i j}e_{i j},
y=\sum\limits_{i,j=1}^{n}b_{i j}e_{i j},\,a_{i j}, b_{i j}\in S(Z(M)),\,i,j=\overline{1, n}
$$
we have
$$
T_\phi(x y)=T_\phi\left(\sum\limits_{i,j=1}^{n}a_{i j}e_{i j}\sum\limits_{k, s=1}^{n}b_{k s}e_{k s}\right)=
T_\phi\left(\sum\limits_{i,j, s=1}^{n}a_{i j}b_{j s}e_{i s}\right)=$$
$$
=
\sum\limits_{i,j, s=1}^{n}\phi(a_{i j}b_{j s})e_{i s}=
\sum\limits_{i,j, s=1}^{n}\phi(a_{i j})\phi(b_{j s})e_{i s}
=
$$
$$
=
\sum\limits_{i,j=1}^{n}\phi(a_{i j})e_{i j}\sum\limits_{k, s=1}^{n}\phi(b_{k s})e_{k s}=T_\phi(x)T_\phi(y),
$$
i.e.
$T_\phi(xy)=T_\phi(x)T_\phi(y).$

 The following property immediately follows from the definition of
 $T_\phi:$

if $\varphi$ and $\phi$ are two automorphisms of $S(Z(M))$ then
 $T_{\phi}\circ
T_{\varphi}=T_{\phi\circ\varphi},$ in particular
  $T^{-1}_\phi=T_{\phi^{-1}}.$

\begin{remark} \label{R1}

 (i) If the automorphism $\phi$ on $S(Z(M))$ is non trivial (i.e. not identical)
then it is clear that $T_\phi$ can not be an inner automorphism on
$M_n(S(Z(M)).$

(ii) It is known \cite[Lemma 1]{KR74} that every (algebraic)
automorphism of $C^\ast$-algebra is automatically norm continuous.
But in our case this is not true in general. Suppose that the
abelian algebra $S(Z(M))$ is represented as $L^0(\Omega, \Sigma,
\mu),$ with a continuous Boolean algebra $\Sigma.$ Then A.G.
Kusraev \cite[Theorem 3.4]{Kus} has proved that $S(Z(M))$ admits a
non trivial band preserving automorphism which is, in particular
$t(M)$-dis\-continuous. Therefore $T_\phi$ gives an example of a
$t(M)$-dis\-continuous automorphism of $E(M).$ In particular,
$T_\phi$ is not inner.
\end{remark}

\begin{proposition}\label{E}If  $M$ is a von Neumann algebra of type
 I$_n,$ then each auto\-morphism  $T$ of  $E(M)$
can be uniquely represented in the form
\begin{equation}
\label{autofin}
T=T_{a}\circ T_{\phi},
\end{equation}
where  $T_{a}$ is an inner automorphism implemented by an element
$a\in E(M),$ and  $T_{\phi}$ is the automorphism of the form
 \eqref{cenfin} generated by an automorphism $\phi$ of the center $S(Z(M)).$
\end{proposition}

\begin{proof}
Let  $\phi$ be the restriction of  $T$ onto the center  $Z(E(M))=S(Z(M)).$
As it was mentioned earlier $\phi$ map  $Z(E(M))$ onto itself, i.e.
$\varphi$ is an automorphism of $Z(E(M)).$
Consider the automorphism  $T_\phi$ defined by
  \eqref{cenfin} and put
$S=T\circ T_{\phi}^{-1}.$
Since  $T$ and  $T_\phi$ coincide on $Z(E(M)),$ one has that
 $S$ is identical on the center  $Z(E(M)),$
 i.e. $S$ is a $Z(E(M))$-linear automorphism of $E(M).$
 By Theorem  \ref{A1}  there exists an invertible element
  $a\in E(M)$ such that $S=T_a,$ i.e. $S(x)=axa^{-1}$ for all $x\in E(M).$ Therefore
$T=S\circ T_\phi=T_a\circ T_\phi.$

Suppose that  $T=T_a\circ T_\phi=T_b\circ T_\varphi$ for $a, b \in E(M)$ and
automorphisms $\phi$ and $\varphi$ of $Z(E(M)).$
Then  $T_b^{-1}\circ T_a=T_\varphi\circ T_\phi^{-1},$ i.e.
$T_{b^{-1}a}=T_{\varphi\circ \phi^{-1}}.$
Since  $T_{b^{-1}a}$ is identical on the center  $Z(E(M))$ of $E(M),$
it follows that $\varphi\circ\phi$ is identical on the center $Z(E(M)),$ i.e.
 $\varphi=\phi.$ Therefore $T_\varphi=T_\phi,$ i.e. $T^{-1}_b\circ T_a=Id$ and hence
 $T_a=T_b.$
\end{proof}

\begin{proposition}\label{F}
Let  $M$ be a von Neumann algebra and let $T:E(M)\rightarrow E(M)$
be an automorphism. If $x\in E(M)$ and its central cover
$c(x)=\textbf{1}$ then  $c(T(x))=\textbf{1}.$
\end{proposition}

\begin{proof}
Let $c(x)=\textbf{1}$ and consider a central projection $z\in
P(Z(M))$ such that    $T(z)=\textbf{1}-c(T(x)).$ Then
$$
T(z x)=T(z)T(x)=(\textbf{1}-c(T(x))c(T(x))T(x)=0
$$
and hence $zx=0.$ Therefore  $zc(x)=0,$ i.e. $z=0.$
This means that  $0=T(0)=\textbf{1}-c(T(x))=\textbf{1},$
i.e. $c(T(x))=\textbf{1}.$
\end{proof}

If $\phi$ is a *-automorphism of $E(M)$ then it is an order
automorphism and hence maps $M$ onto $M.$ But for an arbitrary
automorphism (non adjoint preserving), this not true in general.
For some particular cases one can obtain a positive result.

\begin{proposition}\label{H}
Let  $M$ be an abelian von Neumann algebra and let
 $\phi:E(M)\rightarrow E(M)$ be a  $t(M)$-continuous automorphism.
 Then
$\phi(M)\subseteq M.$
\end{proposition}

\begin{proof} Let  $x\in M$ be a simple element,
i.e.
\[
x=\sum\limits_{i=1}^{n}\lambda_i e_i,
\]
 where  $\lambda_i\in \mathbb{C}, e_i\in P(M), e_i e_j=0, i\neq j, i, j =\overline{1, n}.$
 Let us prove that
 $\phi(x)\in M$ and  $||\phi(x)||_M=||x||_M.$ Since
 $M$ is abelian and $\phi(e_i)^2=\phi(e_i),$ it follows that
 $\phi(e_i)$ is a projection for each  $i=\overline{1, n}.$
Therefore from the equality
 $$
\phi(x)=\sum\limits_{i=1}^{n}\lambda_i \phi(e_i)
$$
we obtain that   $\phi(x)\in M$ and moreover
$$
||\phi(x)||_M=\max\limits_{1\leq i\leq n}|\lambda_i|=||x||_M.
$$

Let now  $x\in M$  be an arbitrary element. Consider
a sequence  of simple elements $\{x_n\}$ in  $M$ which
$t(M)$-converges to $x$
and
$
|x_n|\leq |x|
$
for all $n\in \mathbb{N}.$ Then
$
\phi(x_n)\stackrel{t(M)}\longrightarrow \phi(x)
$
and
$
||\phi(x_n)||_M=||x_n||_M\leq ||x||_M
$
for all $n\in \mathbb{N}.$
Therefore
$
|\phi(x)|\leq ||x||_M\textbf{1},
$
i.e. $\phi(x)\in M.$
\end{proof}

We are now in a position to consider automorphisms of central
extensions for type I$_{\infty}$ von Neumann algebras.

\begin{proposition}\label{J}
Let   $M$ be a von Neumann algebra of type $I_{\infty},$ and let
 $T:E(M)\rightarrow
E(M)$ be an automorphism of the central extension $E(M)$ of $M.$
Then  $T$ is $t(Z(M))$-continuous on $E(Z(M))$ and maps  $Z(M)$
onto itself.
\end{proposition}

\begin{proof} Since    $M$ is of type  I$_{\infty},$ there exists a sequence
of mutually orthogonal abelian projections
  $\{p_n\}_{n=1}^{\infty}$ in  $M$ with central covers equal to  $\textbf{1}.$
For a bounded sequence   $\{a_n\}$ from $Z(M)$
put
$$x=
\sum\limits_{n=1}^{\infty}a_n p_n.$$
Then
\[
x p_n=p_n x=a_n p_n
\]
for all $n\in \mathbb{N}.$

Now let $T$ be an automorphism of $E(M)$ and denote by $\phi$
its restriction onto the center of   $E(M).$
If  $q_n=T(p_n),$   $n\in \mathbb{N},$ then we have
$$T(x p_n)=T(x)T(p_n)=T(x)q_n$$
and
$$
T(x p_n)=T(a _n p_n)=T(a_n)T(p_n)=\phi(a_n)q_n,$$
therefore
\[
T(x)q_n=\phi(a_n)q_n.
\]

For the center-valued norm  $\|\cdot\|$
on   $E(M)$ (see  (\ref{norm})) we have
 $$
 \|q_n||||T(x)\|\geq\|q_n T(x)\|=
\|\phi(a_n)q_n\|=
|\phi(a_n)|||q_n\|,$$
i.e.
 $$
 \|q_n||||T(x)\|\geq
|\phi(a_n)|||q_n\|.$$ Since  $c(q_n)=c(p_n)=\textbf{1}$
(Proposition  \ref{F}) the latter inequality implies that
\begin{equation}
\label{ineq}
||T(x)\|\geq
|\phi(a_n)|.
\end{equation}

Let us show    that $\phi$ is  $t(Z(M))$-continuous on $E(Z(M)).$
If we suppose the opposite, then there
  exists a bounded sequence   $\{a_n\}$ in   $Z(M)$ such that
$\{\phi(a_n)\}$ is not  $t(Z(M))$-bounded, which contradicts
 (\ref{ineq}). Thus   $\phi$ is  $t(Z(M))$-continuous and
 Proposition  \ref{H}  implies that   $T$ maps  $Z(M)$ onto itself.
  \end{proof}

\begin{remark} \label{R3}
The $t(Z(M))$-continuity of $T$ on the center $E(Z(M))$ easily
implies that the restriction of $T$ on $E(Z(M))$ and hence on
$Z(M)$ is a *-automorphism  (cf. \cite[Lemma 1]{KR74}).
\end{remark}

Now we are going to  show that similar to the case of type I$_n$
$(n\in \mathbb{N})$ von Neumann algebras, automorphisms of the
algebras  $E(M)$ for homogeneous type I$_\alpha$ von Neumann
algebras ($\alpha$ is an infinite cardinal numbers) also can be
represented in the form  (\ref{autofin}).

Suppose that  $\phi:Z(M)\rightarrow Z(M)$ is an automorphism.
According to  \cite[Theorem 1]{Kap52} $\phi$ can be extended to a
*-automorphism of  $M,$ which we denote by
  $T_{\phi}.$ Since each *-automorphism is an order isomorphism
  and each hermitian element of $E(M)$ is an order limit of hermitian
  elements from $M,$
  we can naturally extend  $T_{\phi}$ to a  *-automorphism of $E(M).$

\begin{theorem}\label{A2}
If  $M$ is a type
I$_\alpha$ von Neumann algebra, where
 $\alpha$ is an infinite cardinal number,
 then each automorphism  $T$ on  $E(M)$
can be uniquely represented as
\[
T=T_{a}\circ T_{\phi},
\]
where  $T_{a}$ is an inner automorphism implemented by an element
$a\in E(M)$ and  $T_{\phi}$ is an *-automorphism, generated by an automorphism
 $\phi$ of the center  $Z(M)$ as above.
\end{theorem}

\begin{proof} Let   $M$ be an automorphism of
$E(M)$ where $M$ is a type  I$_\alpha$ von Neumann algebra with the center
 $Z(M).$
If $\phi$ is the restriction of $T$ onto the center
 $S(Z(M))$
of $E(M),$ then by Proposition
 \ref{J}   $\phi$ maps  $Z(M)$ onto itself.
 By  \cite[Theorem 1]{Kap52} as above $\phi$ can be extended
 to a  *-automorphism of  $E(M).$
   Now similar to the Proposition
    \ref{E} there exists an element
    $a\in E(M)$ such that  $T=T_{a}\circ T_{\phi}$
and this representation is unique.
\end{proof}

\begin{proposition}\label{G}
Let   $M$ and  $N$ be von Neumann algebras of type
 I and suppose that $M$ is homogeneous of type I$_\alpha.$
  If there exists an isomorphism
 (not necessary  *-isomorphism)
$T$ from $E(M)$ onto $E(N)$ then $N$ is also of type I$_\alpha.$
\end{proposition}

\begin{proof} Let $z_N$ be  a central projection in $N$ such
that $z_N N$ is of type I$_\beta,$ where $\beta$ is a cardinal
number. Take a central projection $z_M$ in $M$ such that
$T(z_M)=z_N.$ Replacing  $M$ and $N$ by $z_M M$ and $z_N N$
respectively we may assume that $z_M=\textbf{1}_M,$
$z_N=\textbf{1}_N.$

 Let  $\{p_i\}_{i\in I}$ (respectively  $\{e_j\}_{j\in j}$ )
be a family of mutually equivalent and orthogonal abelian
projections in $M$ (respectively in  $N$) with
$\bigvee\limits_{i\in I}p_i=\textbf{1}_M,$ (respectively
$\bigvee\limits_{j\in J}e_j=\textbf{1}_N,$) where
 $|I|=\alpha, |J|=\beta.$ It is clear that $c(p_i)=\textbf{1}_M$ for
 all $i\in I.$

 Then $q_i=T(p_i)$ is an idempotent ($q_i^2=q_i$) but
 not a projection in general. Let   $f_i=s_l(q_i)$ be  the left projection of the idempotent
  $q_i.$ Since $f_i$ is the projection onto  the range of the idempotent $q_i$  we have that
   $q_i f_i=f_i,$
  i.e. $f_i q_i f_i=f_i,$  and moreover  $c(f_i)=\textbf{1}_N,$
  because $c(q_i)=\textbf{1}_N$ (see Proposition \ref{F}). The equalities
  \[
q_iE(N)q_i=T(p_i E(M)p_i)=T(Z(E(M))p_i)=E(Z(N))q_i,
\]
 imply that for each $x\in E(N)$  there exists $a_x\in E(Z(N))$ such that
 $q_i xq_i=a_x q_i.$

 Now we show that $f_i$ is an abelian projection. For $x\in E(N)$ and each $f_i$ there exist $a_i\in E(Z(N))$
 such that
 $$
q_i f_i x f_i q_i=a_i q_i.
 $$
Thus
 $$
f_i x f_i=(f_i  q_i f_i) x (f_i q_i f_i)= f_i  (q_i f_i x f_i q_i)
f_i=f_i a_i q_i f_i=a_i f_i q_i f_i=a_i f_i,$$ i.e.
 $f_iE(N)f_i=E(Z(N))f_i.$ This means that
$f_i$ is an abelian projection.

Case  1. $\alpha$ and $\beta$ are finite. Let  $\Phi$ be a normed center-valued
trace on  $N.$
Then
\[
\textbf{1}_N=\Phi(\textbf{1}_N)=\sum\limits_{i\in I}\Phi(q_i)=\alpha \Phi(q_1)=\alpha \Phi(f_1q_1)=
\alpha \Phi(f_1q_1f_1)=\alpha \Phi(f_1).
\]
Since $N$ is of type $I_\beta,$
we have that
\[
\textbf{1}_N=\beta \Phi(f_1).
\]
Therefore
$\alpha=\beta.$

Case  2. $\alpha$ and $\beta$ are infinite. For a faithful normal
semi-finite trace  $\tau$ on  $N$ put
\[
\tau_i(x)=\tau(f_i x), x \in N.
\]
For each  $i\in I$ set
\[
J_i=\{j\in J: \tau_i(e_j)\neq 0.\}
\]
Since $\{e_j\}$ is an orthogonal family, one has that
 $J_i$  is countable for each  $i\in I.$

Suppose that there exists  $j\in J$ such that
 $\tau_i(e_j)=0$ for all $i\in I.$
 Since $\tau(f_i e_jf_i)=\tau(f_i e_j)=\tau_i(e_j)=0,$ we obtain that
  $f_ie_j f_i=0.$ But from
  $$
 0= f_ie_j f_i=f_i e_j e_j f_i= f_i e_j (f_i e_j)^{\ast}
  $$it follows that
    $f_i e_j=0$ for all  $i\in I.$ And since  $\bigvee\limits_{i\in I}f_i=\textbf{1}_N,$
this implies that $e_j=0$ -- a contradiction. Therefore given any
 $j\in J$ there exists   $i\in I$ such that
 $\tau_i(e_j)\neq 0,$ i.e. $j\in J_i.$ Hence
\[
J=\bigcup\limits_{i\in I}J_i,
\]
i.e.
\[
\beta\leq \alpha\aleph_0,
\]
therefore    $\beta\leq\alpha.$ Similarly $\alpha=\beta.$

This means that every homogeneous direct summand of the von
Neumann algebra $N$ is of type I$_\alpha,$ i.e. $N$ itself is
homogeneous of type I$_\alpha.$
\end{proof}

It is well-known \cite{Sak} that if   $M$ is  an arbitrary von
Neumann algebra of type  I with the center  $Z(M)$ then there
exists   an orthogonal family of central projections
$\{z_\alpha\}_{\alpha\in J}$ in  $M$ with $\sup\limits_{\alpha\in
J}z_{\alpha}=\textbf{1}$ such that $M$ is $\ast$-isomorphic to the
$C^{*}$-product of von Neumann algebras $z_\alpha M$ of type
I$_{\alpha}, \alpha\in J,$  i.e.
$$M\cong\bigoplus\limits_{\alpha\in J}z_\alpha M.$$
In this case by definition  of the central extension we have that
$$E(M)=\prod\limits_{\alpha\in J}E(z_\alpha M).$$

Suppose that  $T$ is an automorphism of
 $E(M)$ and  $\phi$ is its restriction onto the center
  $E(Z(M)).$ Let us show that  $T$ maps each
$z_\alpha E(M)\cong E(z_\alpha M)$ onto itself. The automorphism
$T$ maps  $z_\alpha E(M)$ onto  $T(z_\alpha)E(M).$ From
Proposition \ref{G} it follows that the von Neumann algebra
$T(z_\alpha) M$ is of type I$_\alpha.$ Thus $T(z_\alpha)\leq
z_\alpha.$ Suppose that $z_\alpha'=z_\alpha-T(z_\alpha)\neq 0.$ By
Proposition
 \ref{G} we have that $T^{-1}(z_\alpha') M$  is of type  I$_\alpha,$
i.e.
$$
0\neq z_\alpha''=T^{-1}(z_\alpha')\leq z_\alpha.
$$
On other hand
$$
T(z_\alpha z_\alpha'')=T(z_\alpha)T(z_\alpha'')=
T(z_\alpha)z_\alpha'=T(z_\alpha)(z_\alpha-T(z_\alpha))=
T(z_\alpha)-T(z_\alpha)=0,
$$
i.e.
 $z_\alpha z_\alpha''=0.$ Therefore since $z_\alpha''\leq z_\alpha$ we have that
   $z_\alpha''=0,$ -- a contradictions with the inequality
 $z_\alpha''\neq 0.$ Hence
$z_\alpha'=0,$ i.e. $T(z_\alpha)=z_\alpha.$

Therefore $\phi$ generates an automorphism
 $\phi_\alpha$ on each
$z_\alpha S(Z(M))\cong Z(E(z_\alpha M)),$ for   $\alpha\in J.$
Let  $T_{\phi_{\alpha}}$ be the automorphism of $z_\alpha E(M)$
generated by  $\phi_\alpha, \alpha\in J.$ Put
\begin{equation}
\label{auto6}
T_\phi\left(\{x_\alpha\}_{\alpha\in J}\right)=\{T_{\phi_\alpha}(x_\alpha)\},\,\{x_\alpha\}_{\alpha\in J}\in E(M).
\end{equation}
Then  $T_{\phi}$ is an automorphism of  $E(M).$

Now we can state the main result of the present paper.

\begin{theorem}\label{A3}
If  $M$ is a type I von Neumann algebra, then
each automorphism $T$ of  $E(M)$
can be uniquely represented in the form
\[
T=T_{a}\circ T_{\phi},
\]
where  $T_{a}$ is an inner automorphisms implemented by an element
$a\in E(M)$ and  $T_{\phi}$ is an automorphism of the form  (\ref{auto6}).
\end{theorem}

\begin{proof} Let  $T$ be an automorphism of  $E(M)$ and
  $\phi$
be its restriction on   $Z(E(M))$ -- the center of  $E(M).$
Consider the automorphism $T_{\phi}$ on  $E(M)$ generated by the automorphism
 $\phi$ as in  \eqref{auto6} above. Similar to the proof of Proposition \ref{E}  we find an element
$a\in E(M)$ such that  $T=T_{a}\circ T_{\phi}$ and show that this
representation is unique.
\end{proof}

Recall  \cite{Gut}, \cite{Kus} that an operator  $T:E(M)\rightarrow E(M)$
is called \emph{band preserving}  if  $T(zx)=zT(x)$ for all  $z\in P(Z(M)),$
$x\in E(M).$

Proposition \ref{J} and Theorem \ref{A3} imply the following
result which is an analogue of \cite[Theorem 5, Remark A]{KR74}
giving a sufficient condition for innerness of algebraic
automorphisms.

\begin{corollary}\label{Co}
If $M$ is a von Neumann algebra of type
 I$_\infty$ then each band preserving automorphism of
  $E(M)$ is inner.
\end{corollary}

\begin{proof} Let $\phi$ be the  restriction of  $T$ onto  $E(Z(M)).$
Since  $T$ is band preserving it follows that  $\phi$ acts
identically on the  simple elements from  $Z(M).$  Proposition
\ref{J} implies  that $\phi$ is  $t(Z(M))$-continuous. Hence
$\phi$ is identical on the whole $S(Z(M))=E(Z(M))$ and therefore
by Theorem \ref{A3} $T$ is an inner automorphism.
\end{proof}

\begin{remark} It is clear that the conditions of the
above Corollary is also necessary for the innerness of
automorphisms of $E(M).$
\end{remark}

\section*{Acknowledgments}

The second and the third named authors would like to acknowledge
the hospitality of the "Institut f\"{u}r Angewandte Mathematik",
Universit\"{a}t Bonn (Germany). This work is supported in part by
the DFG AL 214/36-1 project (Germany).

\end{document}